\documentclass[12pt]{amsart}

\pdfoutput=1

\usepackage{a4wide}
\usepackage{color}

\usepackage{esint,amssymb}
\usepackage{graphicx}
\usepackage[english,french]{babel} 

\usepackage[colorlinks=true, pdfstartview=FitV, linkcolor=blue, citecolor=blue, urlcolor=blue,pagebackref=false]{hyperref}
\usepackage{microtype}

\usepackage{bm}
\usepackage{scalerel} %for scaling the boxes 
\usepackage{dsfont}
\usepackage[font={footnotesize}]{caption}

\definecolor{darkgreen}{rgb}{0,0.5,0}
\definecolor{darkblue}{rgb}{0,0,0.7}
\definecolor{darkred}{rgb}{0.9,0.1,0.1}

\newtheorem{theo}{Theorem}
\newtheorem{prop}{Proposition}[section]
\newtheorem{lem}[prop]{Lemma}
\newtheorem{coro}[prop]{Corollary}
\newtheorem{remark}[prop]{Remark}

\theoremstyle{plain}

\theoremstyle{definition}

\numberwithin{equation}{section}
\numberwithin{prop}{section}
\newcommand{\R}{\mathbb{R}}

\newcommand{\ep}{\varepsilon}

\newcommand{\g}{\mathbf{g}}

\def \be{\begin{equation}}
\def \ee{\end{equation}}

\def \t0{\rightarrow 0} % Vers z√©ro
\def \hal{\frac{1}{2}}

\def \supp{\mathrm{supp }} % Support
\def \div{\mathrm{div} \,} % Divergence
\def \1{\mathbf{1}} % Fonction caract√©ristique

\def \ep{\varepsilon}
\def \dist{\mathrm{dist}}

\def\cd{\mathsf{c_{\d}}}

\def\({\left(}
\def\){\right)}
\def\nab{\nabla}

\renewcommand{\subset}{\subseteq}

\renewcommand{\subset}{\subseteq}

\renewcommand{\div}{\divg}
\newcommand{\indic}{\mathds{1}}
\newcommand{\indc}{\mathds{1}}

\renewcommand{\hat}{\widehat}

\newcommand{\Rd}{\R^\d}

\def\Xint#1{\mathchoice
   {\XXint\displaystyle\textstyle{#1}}%
   {\XXint\textstyle\scriptstyle{#1}}%
   {\XXint\scriptstyle\scriptscriptstyle{#1}}%
   {\XXint\scriptscriptstyle\scriptscriptstyle{#1}}%
   \!\int}
\def\XXint#1#2#3{{\setbox0=\hbox{$#1{#2#3}{\int}$}
     \vcenter{\hbox{$#2#3$}}\kern-.5\wd0}}
\def\dashint{\Xint-}

\def\HN{\mathcal{H}_N}

\def\meseq{\mu_V}

\def\pa{\partial}

\def \mut{\overline{\mu}_{t}}

\def\g{\mathsf{g}}
\def\d{\mathsf{d}}

\def\muv{\mu_\infty}

\def\indic{\mathbf{1}}

\def\div{\mathrm{div}\,}

\def\mub{\mu_{\beta}}
\def\mut{\mu_{\beta}}

\makeatletter
\newenvironment{abstracts}{%
  \ifx\maketitle\relax
    \ClassWarning{\@classname}{Abstract should precede
      \protect\maketitle\space in AMS document classes; reported}%
  \fi
  \global\setbox\abstractbox=\vtop \bgroup
    \normalfont\Small
    \list{}{\labelwidth\z@
      \leftmargin3pc \rightmargin\leftmargin
      \listparindent\normalparindent \itemindent\z@
      \parsep\z@ \@plus\p@
      
      \itemsep\medskipamount
    }%
}{%
  \endlist\egroup
  \ifx\@setabstract\relax \@setabstracta \fi
}

\newcommand{\abstractin}[1]{%
  \otherlanguage{#1}%
  \item[\hskip\labelsep\scshape\abstractname.]%
}
\makeatother

\begin{document}

\title{Thermal approximation of the equilibrium measure and obstacle problem}

\begin{abstracts}

\abstractin{english}
We consider the probability measure minimizing a free energy functional equal to the sum of a Coulomb interaction, a confinement potential and an entropy term, which arises in the statistical mechanics of Coulomb gases.  In the limit where the inverse temperature~$\beta$ tends to $\infty$ the entropy term disappears and the  measure, which we call the ``thermal equilibrium measure'' tends to the well-known equilibrium measure, which can also be interpreted as a solution to the classical obstacle problem. We provide quantitative estimates on the convergence of the  thermal equilibrium measure to the equilibrium measure in strong norms in the bulk of the latter, with a sequence of  explicit correction terms in powers of $\beta^{-1}$, as well as an analysis of the tail after the boundary layer of size $\beta^{-1/2}$.

\smallskip

\abstractin{french}
On consid\`ere la mesure de probabilit\'e qui minimise une \'energie libre \'egale \`a la somme d'une interaction  coulombienne, d'un potentiel de confinement et d'un terme d'entropie, et qui 
appara\^{\i}t
en m\'ecanique statistique des gaz de Coulomb. Dans la limite o\`u la temp\'erature inverse $\beta$ tend vers l'infini, le terme d'entropie 
dispara\^{\i}t 
et la mesure, que l'on appelle ``mesure d'\'equilibre thermique'', tend vers la mesure d'\'equilibre habituelle qui peut \'egalement  \^etre interpr\'et\'ee comme solution du probl\`eme de l'obstacle classique. On obtient des estim\'ees quantitatives de convergence de la mesure d'\'equilibre thermique vers la mesure d'\'equilibre dans des normes fortes \`a l'int\'erieur du support de cette derni\`ere, avec une s\'erie de termes correctifs explicites en puissances inverses de $\beta$, de m\^eme qu'une analyse des queues apparaissant apr\`es une couche limite de taille $\beta^{-1/2}$.
\end{abstracts}

\selectlanguage{english}

%\begin{abstract} 
%We consider the probability measure minimizing a free energy functional equal to the sum of a Coulomb interaction, a confinement potential and an entropy term, which arises in the statistical mechanics of Coulomb gases.  In the limit where the inverse temperature~$\beta$ tends to $\infty$ the entropy term disappears and the  measure, which we call the ``thermal equilibrium measure'' tends to the well-known equilibrium measure, which can also be interpreted as a solution to the classical obstacle problem. We provide quantitative estimates on the convergence of the  thermal equilibrium measure to the equilibrium measure in strong norms in the bulk of the latter, with a sequence of  explicit correction terms in powers asfdasdfsdaf $\beta^{-1}$, as well as an analysis of the tail after the boundary layer of size $\beta^{-1/2}$. 
%\end{abstract}

\author[S. Armstrong]{Scott Armstrong}
\address[S. Armstrong]{Courant Institute of Mathematical Sciences, New York University, 251 Mercer St., New York, NY 10012}
\email{scotta@cims.nyu.edu}

\author[S. Serfaty]{Sylvia Serfaty}
\address[S. Serfaty]{Courant Institute of Mathematical Sciences, New York University, 251 Mercer St., New York, NY 10012}
\email{serfaty@cims.nyu.edu}

\subjclass[2010]{}
\date{\today}

\maketitle

\section{Introduction} 
\subsection{Setting of the problem}
The Coulomb gas   is a system of points in $\Rd$ with pairwise interaction $\g$ defined by 
\begin{align*}
\g(x) :=
\left\{ 
\begin{aligned}
& - \log |x| & \text{if} & \ \d=2, \\
& |x|^{2-\d} & \mbox{if} & \ \d >2,
\end{aligned}
\right.
\end{align*}
and  an external  (or confinining) potential (or field) $V$, so that the total energy of the system 
of $N$ point at locations $x_1, \dots, x_N$ is given by
\be\label{HN}\HN (x_1,\dots, x_N)
:= 
\hal \sum_{1 \leq i \neq  j \leq N} \g(x_i-x_j) + N\sum_{i=1}^N  V(x_i).\ee
Here, the strength of the external potential $V$ has been scaled so that the potential energy is of the same order as the interaction energy. 
In  the limit $N \to \infty$,  called the ``mean field limit,'' one is led to minimizing among probability measures
the (mean-field) energy 
 \begin{equation}
 \label{defE}
 \mathcal E(\mu)
 :=
 \hal \int_{\R^\d\times \R^\d} \g(x-y) d\mu(x) d\mu(y)+ \int_{\R^\d} V(x) d\mu(x).
 \end{equation}
Here $\mu$ should be thought of as the limit as $N\to \infty$ of the empirical measures $\frac{1}{N}\sum_{i=1}^N \delta_{x_i}$.

\smallskip

It is well known that if~$V$ grows sufficiently fast at infinity, problem \eqref{defE} has  a unique minimizer among probability measures, called the {\it equilibrium measure}, or the {\it Frostman equilibrium measure}, see for instance~\cite{safftotik} for the two-dimensional case. This measure will be denoted~$\muv$.
It is well-known that minimizers of 
\eqref{HN} converge as~$N\to \infty$ to~$\muv$ in the sense of measures (see~\cite{choquet} or~\cite[Chap. 2]{ln}).

The equilibrium measure $\muv$  is typically compactly supported and characterized by the fact that    there exists a constant $c_\infty$ such that letting 
\begin{equation}\label{defzeta}
\zeta(x):= \int_{\R^\d} \g(x-y) d\muv(y) + V(x)-c_\infty,
\end{equation}
 we have  $\zeta=0$ q.e.~in $\supp \, \muv$ and $\zeta \ge 0$ q.e.~where q.e.~is the abbreviation of ``quasi-everywhere'' which means except on a set of zero capacity.

\smallskip

This way we can see that $\muv$ can be interpreted in terms of the classical obstacle problem. Using the notation 
\be \label{hmu} 
h^{\mu}(x)
:= 
\int_{\R^\d} \g(x-y)d\mu(y)
\ee
the function $h^{\muv}$ satisfies $-\Delta h^{\muv}=\cd \muv$, where 
\begin{equation}
\cd := 
\left\{
\begin{aligned}
& 2\pi & \mbox{if} & \ \d=2,\\
& \d(\d-2) |B_1| & \mbox{if} & \ \d>2, 
\end{aligned}
\right. 
\end{equation}
is the constant for which $-\Delta \g=\cd \delta_0$.
By the above properties on $\zeta$ it holds that 
\be \label{op}
\min \( h^{\muv}+ V-c_\infty,- \Delta h^{\muv}\) =0
\quad \mbox{in} \ \Rd, 
\ee
which is precisely the equation for the solution to the classical obstacle problem in whole space with obstacle $c_\infty-V$.
For more details about this correspondance between equilibrium measure and obstacle problem, one can see for instance \cite[Chap. 2]{ln},  \cite{asz} and references therein. The dependence of $\muv$ in $V$ has been previously examined in this full space context in \cite{serser}.

\smallskip

The Gibbs measure corresponding to a Coulomb gas at inverse temperature $\beta$ is 
\be\label{mesgibbs}\exp\(- \frac{\beta}{N} \HN(x_1, \dots, x_N) \) dx_1\dots dx_N.\ee
Different normalizations of $\beta$ with respect to $N$ can be chosen, the specific above choice with $1/N$ in front of the energy leads in the mean-fied limit $N \to \infty$ to 
a minimization problem with an added entropy term of the form: 
\be 
 \label{1.9}
\mathcal{E}_{\beta} (\mu):= \mathcal{E} (\mu) + \frac{1}{\beta } \int_{\R^\d} \mu\log \mu,
\ee
see for instance \cite{kiessling,messerspohn,CLMP,bodgui}.
Again \eqref{1.9} should be minimized among probability measures, and if $V$ grows sufficiently fast, it has a unique solution $\mub$ which we will call the {\it thermal equilibrium measure}.   The functional \eqref{1.9} can also be seen as the free energy associated to the McKean-Vlasov equation which is its Wasserstein gradient flow, see for instance \cite{hm} and references therein.

On the other hand, a natural normalization for the energy and temperature in \eqref{mesgibbs} is shown in \cite{ls,as} to be 
\be\label{gibbs}
\exp\(- \beta N^{\frac2\d-1} \HN(x_1, \dots, x_N) \) dx_1\dots dx_N,\ee 
it is natural as $\beta $ fixed is then shown to be the temperature choice that leads to a competition at the microsopic scale between interaction energy and entropy.
This is in particular the normalization most studied in dimension two where~$\beta=2$ then corresponds to the famous determinantal case of the Ginibre ensemble.
This choice, for which~$\beta$ can still be considered to depend on~$N$, then leads in the mean-field limit to minimizing
\be \label{1.10}\mathcal{E}_{\beta} (\mu):= \mathcal{E} (\mu) + \frac{1}{\beta N^{\frac2\d}} \int_{\R^\d} \mu\log \mu\ee
in place of~\eqref{1.9}. In other words it leads to considering the regime where~$\beta$ in~\eqref{1.9} tends to~$\infty$ as~$N \to \infty$, and thus formally to minimizing just~\eqref{defE}.
In~\cite{as}, we showed however that, compared to the usual equilibrium measure minimizing~\eqref{defE}, the thermal equilibrium measure still provides a more precise description of a Coulomb gas, even for the regime with~$\beta$ in~\eqref{1.10} of order one, equivalently~$\beta $ of order~$N^{2/\d}$ in~\eqref{1.9}.

%This is of course even more true when one wishes to study Coulomb gases in large temperature regimes. 

\smallskip

In this paper we thus focus on the regime $\beta \gg 1$ in \eqref{1.9}, where one expects $\mub \to \muv$. This can also be seen as a way to smoothly 
approximate the obstacle problem solution.  The goal of this short paper is to specify how $\mub$ is close to $\muv$ and $h^{\mub}$ to $h^{\muv}$, which we will do in $C^k$  norms.
The quantitative estimates we provide  are crucially used in the papers~\cite{as,s} and allow to treat possibly quite large temperature regimes in \eqref{gibbs} (note that large temperature regimes for Coulomb or log gases have started to gain interest quite recently, see~\cite{RSY2,as,hardylambert}).

We note that this question, although quite natural, does not seem to have been fully answered in the literature,
the only results that we are aware  of  are less precise, they are  those in~\cite{kiessling} which consider the two-dimensional case with no external potential, and~\cite{RSY2} which provide some   results in  the particular case~$V(x)=|x|^2$, and finally the work~\cite{berman} motivated by~K\"ahler geometry, which proves an~$L^\infty$ bound on the difference of~$h^{\mub}$ and~$h^{\muv}$ a bit weaker than~\eqref{ber} (with an extra $\log \beta $ factor) but in the compact setting of a manifold. There were also explicit formulae for the one-dimensional logarithmic case (related but slightly out of our scope) and quadratic potential in~\cite{abg}.

\smallskip

By contrast with $\muv$, $\mub$ is not compactly supported, but always positive in $\R^\d$ and regular. In fact $h^{\mub}$ defined as in \eqref{hmu} solves the  PDE
\be \label{eqhmub}
h^{\mub} +V + \frac{1}{\beta} \log \mub= c_{\beta },\ee
for some constant $c_\beta$.
Taking the Laplacian of that equation leads to a PDE on $\log \mub$ with notoriously delicate exponential nonlinearity
\be \label{eqmub}
 \Delta \log \mub = \beta (  \cd \mub-\Delta V ).\ee
 Instead of studying this equation directly, we observe for the first time that when 
subtracting two such equations (with possible error term) with solutions $\mu$ and $\nu$ respectively, the quotient $u=\mu/\nu-1$ rewrites nicely as a divergence form equation 
\be \label{divfor}\div \frac{\nab u}{1+u}= \beta \mu u+ \mathrm{error}\ee
for which elliptic regularity theory is readily applicable as soon as $u$ is small enough. This allows to obtain corrections to arbitrary order of the approximation $\mub\simeq \muv$, see \eqref{corrections} below. In fact our proofs only use maximum principle-based arguments and regularity theory, and do not require going through energy estimates.

Finally, we comment that the other extreme regime $\beta\to 0$ is easier to treat. We can formally expect the interaction energy to become negligible and we are then led to minimizing, among probability measures, the quantity
$$\int_{\R^\d} Vd\mu + \frac{1}{\beta} \int_{\R^\d} \mu \log \mu,$$
the minimizer of which is 
$\mu= \frac{e^{-\beta V}}{\int e^{-\beta V}},$
%This is summarized in Theorem \ref{th3}, which again generalizes the results of
see \cite{RSY2}.
\subsection{Assumptions and results}
We let $\Sigma:=\supp\,  \muv$
and assume that $\pa \Sigma \in C^{1,1}$. 
 Note that it was very recently established in \cite{frs} that this holds  generically with respect to $V$.
We assume in addition 
\be \label{assumpV1} V \in C^2\ee
\be\label{assumpV2} \begin{cases} V\to +\infty \  \  \text{as} \ |x|\to \infty & \text{if }\ \d\ge 3\\ 
 \lim_{|x|\to \infty} (V+\g ) =+\infty& \text{if } \ \d=2,
\end{cases}\ee
\be
\label{assumpV3} 
\left\{
\begin{aligned}
& \int_{|x|\ge 1} \exp\(-\frac{\beta}{2}  V(x) \)dx<\infty, 
& \mbox{if} & \ \d\geq 3, \\
& \int_{|x|\ge 1} e^{-\frac{\beta}{2}  (V(x)-\log |x| )} \,dx
+ 
\int_{|x|\ge 1} e^{ -\beta  (V(x)-\log |x| ) }|x| \log^2 |x|\, dx<\infty
& \text{if} & \ \d=2,
\end{aligned}
\right. 
\ee
and
\be \label{assumpV4} \Delta V \ge\alpha>0 \quad \text{in a neighborhood of }\, \Sigma.\ee
%Note that \eqref{assumpV1} and \eqref{assumpV2} imply that $V$ is bounded below.

Observe that 
\begin{equation} \label{defomega}
\Sigma \subset \{\zeta = 0\}.
\end{equation}
The set $\{\zeta=0\}$ is called the {\it contact set} or {\it coincidence set} of the obstacle problem, and~$\Sigma$ is the set in which the obstacle is \textit{active}, sometimes called the {\it droplet}.  The assumption~\eqref{assumpV4} ensures that these coincide. Note that~$h^{\muv}=c_\infty-V$ in~$\{\zeta=0\}$,  hence the density satisfies  
$$\muv= \frac{\Delta V}{\cd}\indc_{\Sigma}.$$
Thanks to this connection, the regularity of $\meseq$ and of $\Sigma$ can be known by the standard regularity theory for the classical obstacle problem~\cite{caffarelli} (see also~\cite{serser} for the formulated in the whole space).

Since we assume  $\partial \Sigma \in C^{1,1}$ (which rules out boundary cusps), \eqref{assumpV2} and \eqref{assumpV4}, by standard results on the obstacle problem \cite{caffarelli}, we have 
\be \label{caf}
\zeta(x) \ge \alpha \,\dist(x, \Sigma)^2\quad \text{in a neighborhood of } \Sigma,\ee  with $\zeta$ the function of \eqref{defzeta}, and a corresponding upper bound also holds. We now assume in addition that 
\be\label{assumpV5} \zeta(x) \ge \alpha \min(\dist(x, \Sigma)^2, 1),\ee  which amounts, up to changing to constant $\alpha>0$ if necessary, to assume 
 that  the solution to the obstacle problem never gets very close to the obstacle, outside of $\Sigma$. A sufficient condition is for instance that $V$ be strictly convex.

 %can be implied by that of $V$. Let us summarize the known facts :
%\begin{itemize}
%\item If $V$ is  $C^{1,1}$ then the density of the equilibrium measure is given by
%$$
%d\meseq(x) = \frac{1}{\cd} \Delta V(x) \mathbf{1}_{\Sigma}(x) dx.
%$$
%In particular, if $V\in C^{2, \kappa}$ then $\meseq$ has a $C^{0,\kappa}$ density on its support. If $\Delta V>0$ near $\omega$  then $\Sigma $ and $\omega$ coincide.
%\item The points of the boundary $\partial \omega$ of the coincidence set can be either \textit{regular}, i.e. $\partial \omega$ is locally the graph of a $C^{1, \kappa}$ function, or \textit{singular}, i.e. $\partial \omega$ is locally cusp-like (this classification was introduced in \cite{MR1658612}). Singular points are nongeneric and we implicitly assume that they are absent by assuming \ref{H5}, for technical reasons which  might be bypassed.
%\item If $V$ is $C^{3, \kappa}$, then $\partial \omega$ is locally $C^{2, \kappa}$ around each regular point (see \cite[Thm. I]{caffarelli1976smoothness}).
%\item In the setting of a bounded domain with zero Dirichlet boundary condition, if $V$ is strictly convex (which implies that $\Sigma$ and $\omega $ coincide) and of class $C^{k+1,\kappa}$ on $\R^2$, it was shown (see \cite[Section 4]{kinderlehrer1978variational}) that $\Sigma$ is connected and that $\partial \Sigma$ is $C^{k, \kappa}$ with no singular points. 
%Hence any strictly convex potential in $C^{2,\kappa}$ with the growth condition \ref{H2} satisfies our assumptions. 
%\end{itemize}

\begin{theo}
%[The case of large $\beta$]
\label{th1}Assume \eqref{assumpV1}--\eqref{assumpV4} and \eqref{assumpV5}. Then \eqref{1.10} has a unique minimizer $\mub$. 
Moreover, there exists $C(V,\d) >0$ such that, for every $x\in \R^\d$ and $\beta \in (2,\infty)$, we have
\be 
\label{bornelmub0}
0<\mub(x) \le \begin{cases}\min (C, C\exp\left( -\beta (V(x)   -C)\right) & \text{if }  \ \d \ge 3\\
\min (C, C\exp\left( -\beta (V(x) -\log |x|  -C)\right) & \text{if }  \ \d= 2\end{cases}
\ee 
\be \mub(x) > \frac{1}{C} >0 \quad \text{for} \ x\in \Sigma,\ee
\be   \exp\(-\frac{\beta }{C} \dist(x, \Sigma)^2 - C \)  \le \mub(x)\le \exp\(-\frac{\beta}{C} \dist(x, \Sigma)^2 + C\) \ee in a $\beta$-independent neighborhood of $ \Sigma$, 
 \be \label{ber} \|h^{\mub}- c_\beta - h^{\muv}+c_\infty\|_{L^\infty(\Rd)} \le  \frac{C}{\beta},
\ee
 \be \label{borneh10}
 \|\nab (h^{\muv}-h^{\mut})\|_{L^\infty(\R^\d)} \le \frac{C }{\sqrt\beta}
 ,\ee 
 \be \label{mhS}
 \mub( \Sigma^c)\le \frac{C}{ \sqrt\beta}\ee
and
\be \left|\int_{\Sigma^c} \mub\log \mub \right| \le \frac{C}{\sqrt\beta}.\ee

%\be\label{kp0}
%\int_{\R^\d} |\nab (\mut-\muv)| \le C\ee
Let  $m $ be an integer $\ge 2$ such that  $V\in C^{2m,\gamma}$ for some $\gamma\in (0,1]$ and letting $f_{k}$ be defined iteratively by 
\be \label{41} f_0= \frac{1}{\cd}\Delta V,\qquad f_{k+1}=\frac{1}{\cd}\Delta V+ \frac{1}{\beta \cd}\Delta \log f_k,\ee
we have  $f_k \in  C^{2(m-k-1),\gamma} (\Sigma)$ and  for every even integer $n\le 2m-4$ and $0\le \gamma'\le \gamma$,  if $\beta$ is large enough
depending on $m$, we have 
\be \|\mub-f_{m-2-n/2}\|_{C^{n,\gamma'}( \Sigma)} \le C \beta^{\frac{n+\gamma'}{2}}\exp\(- C \log^2 (\beta \dist(x, \partial \Sigma)^2)\)  
+ C\beta^{1+n-m+\frac{\gamma'}{2}} .\ee
\end{theo}
The functions $f_k$ provide a sequence of improving approximations to $\mub$ defined iteratively. Spelling out the iteration we easily find the expansion in powers of $1/\beta$
\be\label{corrections}\mub\simeq \frac{1}{\cd}\Delta V+ \frac{1}{\cd \beta} \Delta \log \frac{\Delta V}{\cd} + \frac{1}{\cd \beta^2}\Delta \( \frac{\Delta \log \frac{ \Delta V}{\cd}}{\Delta V}\)+... \quad \text{inside }   \Sigma\ee
up to an order dictated by the regularity of $V$ and the size of $\beta$.

The relation \eqref{ber} improves in particular the equivalent result in \cite{berman} (a bound in $\frac{\log \beta}{\beta}$), while
\eqref{borneh10} improves on the  energy comparison-based estimate in $1/\sqrt\beta$ given in \cite{RSY2}.
The estimates reveal the natural lengthscale $1/\sqrt\beta$ appearing in the approximation of $\mub$ by $\muv$.

\begin{remark}
Since $h^{\mub}-h^{\muv}$  vanishes at infinity because $\mub$ and $\muv$ are both probability measures, \eqref{ber}  also implies that 
$$|c_\infty-c_\beta|\le \frac{C }{\beta}.$$
It seems difficult to obtain such a precise estimate from energy considerations only.
\end{remark}
\medskip

The rest of the paper is organized as follows: in Section \ref{appa} we check  the existence of a minimizer to $\mathcal E_\beta$ under the assumptions \eqref{assumpV1}--\eqref{assumpV3} and prove a few of its qualitative properties. 
In Section~\ref{sec3} we obtain a first~$L^\infty$ bound on the difference between the solutions to~\eqref{eqhmub} and~\eqref{op} via a comparison principle, and a uniform bound on~$\mub$. This then serves to obtain a lower bound for~$\mub$ inside~$\Sigma$ by a barrier argument in the following section. This in turn leads to the optimal uniform estimates on~$h^{\mub}-h^{\muv}$ in Section~\ref{secopt}. These estimates are then eventually upgraded  in Section~\ref{sec4} to~$C^k$ spaces via the iterative approximation sequence~$f_k$ thanks to DeGiorgi-Schauder elliptic regularity theory applied to \eqref{divfor}. 
\smallskip

{\bf Acknowledgements:} We  thank Stephen Cameron as well as the anonymous referee for their careful reading and many useful comments that helped improve the paper. SA was supported by NSF grant DMS-1700329 and a grant of the NYU-PSL Global Alliance. SS was supported by NSF grant DMS-1700278 and by the Simons Investigator program.

%\begin{theo}[Stability with respect  to $V$]\label{th2}
%Assume $\xi  \in C^{2p}_c(\R^\d)$ is supported in the set $\{\mub\ge \hal \alpha>0\}$ and let $\mub^t$ be the solution  associated to $V+t\xi$.
%Then letting $u= \frac{\mub^t- \frac{t}{\cd}\Delta \xi}{\mub}-1$ we have 
%\be \|u\|_{C^k} \le \exp\( - C \log^2( \beta \dist(x, \supp \xi)^2)\)\ee
%\end{theo}

%\begin{theo}[The case of small $\beta$]
%\label{th3}
%We have 
%\be \|\log \mub+ \beta (V-c_\beta)\|_{C^k_{loc}}\le C \beta .\ee
%\end{theo}
%\cm{can we express the result rather in terms of $\rho_\beta$?}

\section{Existence of a unique solution and first properties} \label{appa}
Throughout the paper,~$C$ denotes a positive constant which depends only on~$V$ and~$\d$ and may vary in each occurrence.

Some of the results of this section may be known, but  we could not find a reference  and therefore we include them here for the convenience of the reader.

\begin{lem}
If \eqref{assumpV1}--\eqref{assumpV3} hold, then $\mathcal E_\beta$ has a unique minimizer.\end{lem}
\begin{proof}Let us first consider $\d\ge 3$.
We may write 
$$\mathcal E_\beta (\mu)= \int_{\R^\d}  \hal V d\mu+ \hal \iint_{\R^\d\times \R^\d}  \g(x-y) d\mu(x)d\mu(y)
+ \int_{\R^\d}\hal V d\mu+\frac{1}{\beta} \mu \log \mu.$$
 The function $x\mapsto \frac{V}{2} x+\frac{1}{\beta} x\log x$ achieves its minimum at $x= \exp\(- \frac{\beta}{2} V - 1\)$ hence 
 we may bound from below the last integral \be\label{li}
  \int_{\R^\d}\hal V d\mu+\frac{1}{\beta} \mu \log \mu \ge -\int_{\R^\d}\frac{1}{\beta} \exp\( -\frac{\beta}{2} V-1\).   \ee
This is finite by \eqref{assumpV3}. On the other hand, since~$V \in C^2$ and $\g \ge 0$ in the case~$\d\ge 3$, we have
 \be \label{bboo} \int_{\R^\d}  \hal V d\mu+\hal  \iint_{\R^\d\times \R^\d}  \g(x-y) d\mu(x)d\mu(y)> -\infty.\ee 
We deduce that $\inf \mathcal E_\beta>-\infty$.
In view of \eqref{assumpV2} we deduce from the above that minimizing sequences are tight and therefore the existence of a minimizer.
 
We now consider the case~$\d=2$. We note that 
$\g(x-y) \ge - C - \log \max (|x|, |y|, 1)$ and thus, by symmetry,
\begin{align}
\label{abooo}
\mathcal E_\beta(\mu) 
& \geq 
\int_{\R^\d} V d\mu - \iint_{|x|\ge |y|} \log \max(|x|,1)d\mu(x) d\mu(y)  - C + \frac{1}{\beta} \int_{\R^\d} \mu \log \mu
\\ \notag &
\geq  \int_{\R^\d} V d\mu -  \int_{\R^\d}( \log |x|)_+ d\mu(x)  - C + \frac{1}{\beta}\int_{\R^\d} \mu \log \mu. 
\end{align}
Arguing as in the case $\d \ge 3$ but with $V$ replaced by $V-(\log|x|)_+$ and using \eqref{assumpV2} and \eqref{assumpV3}, we deduce the existence result for $\d=2$. 

The uniqueness of the minimizing measure is immediate from the strict convexity of the energy functional.
\end{proof}

\begin{lem} \label{lem22}Under the same assumptions, the minimizer $\mub$ of $\mathcal E_\beta$ is positive almost everywhere in $\Rd$, bounded above, locally bounded below, continuous, and satisfies  \eqref{eqhmub} almost everywhere in $\Rd$.
Moreover, 
we have the following asymptotics
\begin{align}
\label{condinfty}
\left\{
\begin{aligned}
& \lim_{|x| \to \infty} \left(\frac{ h^{\mub}(x)}{\log |x|} +1 \right)= 0 &&  \mbox{in} \ \d=2, \\ 
&
\lim_{|x| \to \infty} h^{\mub}(x)=0 && \mbox{in} \ \d > 2.
\end{aligned}
\right.
\end{align}
%and \be\label{decaynab}
%|\nab h^{\mub}|(x) \sim\max(1,\d-2) |x|^{1-\d} \ \quad \text{as} \ |x|\to \infty.
%\ee\cm{utile?}
\end{lem}
\begin{proof}
{\it Step 1.}
We start by showing that $h^{\mub} \in L^1_{\mathrm{loc}}(\R^\d)$.
Let $B$ be a bounded set in $\R^\d$, we have
\begin{equation}\label{hmubb}
\int_B h^{\mub}=\int_{\R^\d}\int_B  \g(x-y) dx d\mub(y).\end{equation}
If $\d \ge 3$ it is then straightforward, by integrability and boundedness of $\g$, that the right-hand side is finite. If $\d=2$, then we first need to show that 
\be
\label{intlogfini}
\int_{\R^2} (\log |x|)_+ d\mub(x)<\infty.
\ee
To do so, we return to \eqref{abooo} and
use that  for each $x \in \R^2 \setminus \{0\}$, the function
\begin{align*}
\phi(\mu):=\( V-(\log |x|)_+\) \mu + \frac{1}{\beta} \mu \log \mu
\end{align*}
is convex and achieves its minimum at $u_0(x)=\exp\(-\beta (V-(\log |x|)_+) -1\)$. Its second derivative is the decreasing function $\phi''(\mu)= \frac{1}{\beta \mu}$.
%We may thus write, on the one hand,
%\be\label{2der1} \phi(\mu)\ge \phi(u_0)+ \frac{1}{2\beta} (\mu-u_0)^2 \min \( \mu^{-1}, u_0^{-1}\)\ee
%and, on the other hand, 
%\be \label{2der2} 
%\phi(\mu) =\phi(u_0)+ \int_{u_0}^\mu \frac{1}{\beta} \log \frac{y}{u_0} dy
%= \phi(u_0)+\frac{1}{\beta}  \( \mu \log \frac{\mu}{u_0}- \mu +u_0\) .\ee
Thus 
\begin{align}
\label{2der1}
\phi(\mu) 
& 
= \phi(u_0)+ \int_{u_0}^\mu \frac{1}{\beta} \log \frac{y}{u_0} dy
\\ \notag & 
= \phi(u_0)+\frac{1}{\beta}  \( \mu \log \frac{\mu}{u_0}- \mu +u_0\) 
\\ \notag & 
\geq 
\phi(u_0)+ \frac{1}{2\beta} (\mu-u_0)^2 \min \( \mu^{-1}, u_0^{-1}\).
\end{align}
Inserting into~\eqref{abooo} and using~\eqref{2der1} %, \eqref{2der2} 
and \eqref{assumpV2}, we obtain 
\begin{align}\label{abo1}
\mathcal E_\beta(\mu) 
&
\geq 
-C + 
\int_{\R^2 }  
\phi(\mu_\beta(x))\,dx
\\ \notag & 
\ge -C
-\frac1\beta \int_{\R^2 }
 \exp\left(-\beta(V-(\log|x|)_+)-1\right)\,dx 
\\ \notag  & \qquad 
+ \int_{\mub(x) \ge (1+|x|) u_0(x)  }\frac{1}{\beta}   \mub(x) \( \log (1+|x|)-1\)\,dx
\\ \notag & \qquad 
+ \frac{1}{2\beta} \int_{ 2 u_0(x)   \le  \mub(x) \le (1+|x|) u_0(x)}  \frac{\(  \mub(x) - u_0(x)\)^2 }{  (1+|x|) u_0(x)} .
\end{align}
 We next write  
\begin{align*}
\lefteqn{
\int_{\R^2} (\log |x|)_+ \mub (x) dx 
} \quad & 
\\ & 
\leq 
\int_{\mub (x)\le  2 u_0(x) }
2(\log |x|)_+ \exp\left(-\beta (V-(\log |x|)_+)-1\right) dx
\\ & \quad 
+\int_{ \mub (x)\ge (1+|x|) u_0(x) }  (\log |x|)_+ \mub (x) dx 
\\ & \quad 
+
\int_{2u_0(x) \le \mub(x) \le (1+|x|)u_0(x) } 
(\log |x|)_+ ( (1+|x|) u_0(x) )^{\hal}  \frac{\mub(x) }{\( (1+|x|) u_0(x)\)^{\hal}} dx.
\end{align*}
The first integral on the right-hand side is finite by~\eqref{assumpV3}, the second is finite by finiteness of the terms in \eqref{abo1}  and the third is seen to be finite by applying the Cauchy-Schwarz inequality and using~\eqref{abo1} and~\eqref{assumpV3}. This yields~\eqref{intlogfini}.
It then follows from \eqref{hmubb}
that $h^{\mub}\in L^1_{\mathrm{loc}}(\R^\d)$ for $\d=2$.
\smallskip

{\it Step 2.} We check that $\mub>0$ except on a set of measure zero, as in~\cite{neri,RSY2}. Assuming by contradiction that $\mub=0$ in a bounded set $S$ of positive measure, let us consider $\frac{\mub+\ep\indic_{S}}{1+\ep|S|}$. Let us   expand out  
\begin{align*}
\lefteqn{
\mathcal{E}_\beta \left( \frac{\mub+\ep\indic_S}{1+\ep |S|} \right)
} \qquad & 
\\ & 
 = \mathcal E_\beta(\mub)
-\ep |S|  \( \iint \g(x-y) d\mu_\beta(x) d\mu_\beta(y) + \int V d\mu_\beta+ \frac{1}{\beta} \mub \log \mub\)  
\\ & \qquad 
+\ep \int_S ( h^{\mub} +V) 
+ \frac{|S|}{\beta} \ep \log \ep + O(\ep^2).
\end{align*}
By~Step 1 and the fact that $S$ is bounded, we have that $\int_S h^{\mub}+ V <\infty$. We deduce that 
\begin{equation*}
\mathcal{E}_\beta \left( \frac{\mub+\ep\indic_S}{1+\ep |S|} \right)
\leq
\mathcal{E}_\beta(\mub) + C\ep + \frac{|S|}{\beta} \ep \log \ep,
\end{equation*}
a contradiction with the minimality of $\mub$ if $|S|>0$ when 
$\ep$ is chosen small enough. \smallskip

{\it Step 3.} 
We next check that \eqref{eqhmub} is satisfied. For every smooth  compactly supported function~$f$ such that~$\int f d\mub=0$ and~$t\in\R$ with $|t|$ sufficiently small, $(1+t f) \mub$ is a probability measure and we may expand 
 $$\mathcal E_\beta(\mub) \le \mathcal E_\beta((1+t f)  \mub) $$
to find 
$$t \int_{\R^\d} (h^{\mub}+ V + \frac{1}{\beta}\log \mub) f d\mub  + O(t^2)\ge 0,$$  where $h^{\mub}$ is defined as in \eqref{hmu} and may  take infinite values.
Since this is true for all small enough $|t|$ and any smooth $f$ with $\int fd\mub=0$,  and since $\mub>0$ almost everywhere, it follows that \eqref{eqhmub} holds almost everywhere, for some constant $c_\beta$.
 \smallskip

{\it Step 4.} Proof of \eqref{condinfty}.
For  dimension~$\d \ge 3$ we have that~$h^{\mub} \ge 0$ hence from~\eqref{eqhmub} and~\eqref{assumpV2} we deduce that~$\mub$ is bounded above.
To prove \eqref{condinfty}, given $\ep>0$, we choose~$R$ such that~$\mub(B_R^c)<\ep$ (which is possible since~$\mub$ is a probability measure). 
We then write 
\begin{align}
\label{docom}
h^{\mub}(x) 
&
= 
\int_{B_R} \g(x-y) d\mub(y) 
+ \int_{B_R^c \cap \{|y-x|\le \eta\}} 
\g(x-y)d\mub(y)
\\ \notag & \qquad 
+ \int_{B_R^c \cap \{|y-x|\ge \eta\}} \g(x-y) d\mub(y).
\end{align}
The first term of the right-hand side tends to zero when~$x\to \infty$ because~$\g$ does, the second term is bounded by~$\|\mub\|_{L^\infty} \int_{B_\eta} \g $ which tends to zero as~$\eta$ tends to zero by integrability of~$\g$ near the origin, and the last term can be bounded by~$\g(\eta) \mub(B_R^c)<\g(\eta) \ep$.
We may then choose~$\eta$ appropriately to make all the three terms be at most~$\ep^{\hal}$ when~$|x|$ is large enough, which proves~\eqref{condinfty} in the case~$\d \ge 3$. 
 % The proof of \eqref{decaynab} is similar, and the result is complete for $\d \ge 3$.

For~$\d=2$, we use $-\log |x-y| \ge - C - \log \max(|x|, |y|,1)$ and~\eqref{intlogfini} to obtain
\begin{align*}
h^{\mub}(x)+ (\log |x|)_+
&
=
\int_{\R^2} \left(
- \log |x-y| + (\log |x|)_+ \right)\,
d\mu_\beta(y) 
\\ & 
\ge   - C  + \int_{|y|\ge |x|} ( (\log |x|)_+ -(\log |y|)_+)  d\mub(y)
\\ & 
\ge - C.
\end{align*}
Therefore $h^{\mub}+(\log |x|)_+$ is bounded below.
Since $V-(\log |x|)_+$ is  also bounded below by~\eqref{assumpV2}, we deduce from~\eqref{eqhmub} that~$\mub$ is bounded above, and then we can finish the proof as in dimension $\d \ge 3$  from the decomposition \eqref{docom}, using~\eqref{intlogfini}. \smallskip

{\it Step 5.} Continuity.
The computations of the previous step starting from \eqref{docom} show that  $h^{\mub}$ is locally  bounded  above, and so is $V$. It then follows from \eqref{eqhmub} that $\mub$ is locally bounded below.
Once we have shown that~$\mub$ is locally bounded above and below by positive constants, we may rewrite \eqref{eqmub} as a  uniformly elliptic equation for $\mub$:
$$\div \frac{\nab \mub}{\mub}= \beta(\cd \mub-\Delta V).$$
By standard elliptic regularity theory (for instance \cite{GT}), we thus deduce that $\mub$ is as regular as $V$, in particular $\mub$ is continuous.
\end{proof}

\section{The comparison principle and upper bound on \texorpdfstring{$\mub$}{mu beta}}\label{sec3}
\subsection{A preliminary lemma}
We will use the following comparison principle for the obstacle problem in the whole plane.
\begin{lem} \label{l.comp}
Suppose that $v,w $ are two continuous function in $\R^2$ which satisfy
\begin{equation}\label{e.compin}
\min\left\{ -\Delta v, v - (c_\infty - V) \right\} \leq 0 \leq \min\left\{ -\Delta w, w - (c_\infty-V) \right\} \quad \mbox{in} \ \R^2
\end{equation}
as well as
\begin{equation}
\label{e.growth}
\limsup_{|x| \to \infty} \, \frac{v(x)}{\log |x|} \leq -1 \leq \liminf_{|x| \to \infty} \, \frac{w(x)}{\log |x|}.
\end{equation}
Then $v \leq w$ in $\R^2$.
\end{lem}

\begin{proof}[{Proof of Lemma~\ref{l.comp}}]
Let $\phi= c_\infty-V$ be the obstacle function.
We may assume without loss of generality that $\phi\leq 0$ (otherwise we may subtract a constant). Then $v \leq 0$ by the maximum principle, since the zero function is a harmonic function which, due to~\eqref{e.growth}, is larger than $v$ in the complement of a bounded set. Moreover, $\min\{ tw,0\}$, with $0 < t \leq1$, satisfies the same assumptions as $w$, and thus it suffices to show that $v \leq tw$ for every $0<t < 1$. In light of this, we may assume that
\begin{equation*}\label{}
\limsup_{|x| \to \infty} \, \frac{v(x)}{\log |x|} \leq -1 <  \liminf_{|x| \to \infty} \, \frac{w(x)}{\log |x|}.
\end{equation*}
In particular, $\{ v > w\}$ is bounded. Observe also that $ \{ v > w\} \subseteq \{ v > \phi \}$. Since $v$ is subharmonic in the latter and $w$ is superharmonic in $\R^2$, we deduce that $v-w$ is subharmonic in $\{ v - w > 0\}$. Assume that this set is nonempty, to get a contradiction. Let $x_0$ be the point at which $v-w$ attains its global maximum,  say $M:= (v-w)(x_0) = \sup_{\R^2} (v-w)$. Then, since $v-w$ is subharmonic at $x_0$, we deduce that it is constant in a neighborhood of $x_0$. In fact, this argument shows that the set $\{ v-w = M\}$ is open; since $v-w$ is continuous, it is also closed. Since $\{ v-w=M\} \neq \emptyset$, we must have that $v-w\equiv M$. Thus $v$ and $w$ are harmonic. Since $v$ is bounded above, it must be constant. This violates the growth condition.
\end{proof}

\subsection{Main proof}
We now turn to the main comparison result of this section.
\begin{lem}\label{lem24}
Let $m_\beta=\sup_{\Rd} \mub$.  If $\beta$ is large enough, we have
\be \label{compp0}
- \frac{\log m_\beta}{\beta} \le h^{\mub}-c_\beta  - ( h^{\mu_\infty}- c_\infty) .\ee
\end{lem}
\begin{proof}
To compare $h^{\mub} -c_\beta $ and $h^{\muv}-c_\infty$ we recall that $h^{\mub} $ satisfies \eqref{eqhmub} while $h^{\muv}$ satisfies  \eqref{op}.
We may write from \eqref{eqhmub} that 
 \be\label{psot}
 h^{\mub}+V-c_\beta +\frac{\log m_\beta}{\beta}\ge 0 \ee
It follows that 
\be \label{110}
\min \(h^{\mub}+V-c_\beta +\frac{\log m_\beta}{\beta}, -\Delta h^{\mub}\) \ge 0.\ee

 In dimension $\d=2$, applying the  comparison principle of Lemma \ref{l.comp}  to $ h^{\mub} +c_\infty-c_\beta+ \frac{m_\beta}{\beta}   $ and $h^{\muv}$ 
 we deduce that 
$$h^{\mub}+c_\infty-c_\beta + \frac{\log m_\beta}{\beta}\ge h^{\muv},$$
which is the desired result.
For  dimension $\d\ge 3$, we first  show that we have \be \label{limif}
\liminf_{|x|\to \infty} \( h^{\mub}+c_\infty-c_\beta + \frac{\log m_\beta}{\beta} \)\ge 0\ee which is  equivalent by \eqref{condinfty} to showing that 
 $c_\infty-c_\beta + \frac{\log m_\beta}{\beta}\ge 0$.
 
To do so, by contradiction assume that $c_\infty-c_\beta + \frac{\log m_\beta}{\beta}<0$ and let us  consider $\psi$ harmonic in $\R^\d \backslash \Sigma$  such that $\psi=0$ on $\pa \Sigma$ and $\psi= c_\infty-c_\beta + \frac{\log m_\beta}{\beta}$ at $\infty$. Because $c_\infty-c_\beta + \frac{\log m_\beta}{\beta}<0$, we have that
$\psi-  \( c_\infty-c_\beta + \frac{\log m_\beta}{\beta}\)$  decays at infinity like the Green's function, i.e. like $|x|^{2-\d}$.
On the other hand, setting 
$$\varphi:= h^{\mub} -h^{\muv} +c_\infty-c_\beta+ \frac{\log m_\beta}{\beta},$$
by \eqref{psot} and \eqref{op} we have
\be \label{propriphi}\left\{ 
\begin{array}{ll}
\varphi  \ge 0 & \text{ in} \  \Sigma\\
\Delta \varphi \le 0 &  \text{in} \ \R^\d\backslash \Sigma\end{array}\right.\ee
 It then follows that 
 $-\Delta (\varphi-\psi) \ge 0$ in $\R^\d\backslash \Sigma$ with $\varphi-\psi \to 0$ at $\infty$ and $\varphi-\psi \ge 0$ on $\pa \Sigma$. Thus by the maximum principle $\varphi-\psi \ge 0$ in $\R^\d \backslash \Sigma$. 
 On the other hand, since $-\int_{\Rd} \Delta \varphi= \cd\int_{\R^\d} \mub-\muv=0$ we also have that $\varphi- \(c-c_\beta+ \frac{\log m_\beta}{\beta}\)$  decays at infinity like $|x|^{1-\d}$.  This, the fact that $\psi-  \( c-c_\beta + \frac{\log m_\beta}{\beta}\)$  decays at infinity like  $|x|^{2-\d}$, and the fact that $\varphi \ge \psi$ bring a contradiction, which shows that $\liminf_{|x|\to \infty} \varphi \ge 0$.
Since \eqref{propriphi} holds in any case, we then deduce by the maximum principle that $\varphi \ge 0$ in all $\R^\d$, which is the desired result.
\end{proof}
We deduce the following  bounds on $\mub$.
\begin{lem}
\label{lub}
For every $x\in \R^\d$ and $\beta \ge 1$, we have 
\be 
\label{bornelmub}
0<\mub(x) \le \begin{cases} \min (C, C\exp\left( -\beta (V(x)-C)\right) & \text{for} \ \d\ge 3\\
\min(C, C \exp\(- \beta (V(x)-  \log |x|-C) \)  & \text{for} \ \d=2. \end{cases}
\ee 
\end{lem}
\begin{proof}
Let us now turn to the upper bound.
With the result of \eqref{compp0} and bounds on $h^{\muv}$, we have
$$h^{\mub}(x) -c_\beta \ge -\max(1, \log |x|)\indc_{\d=2}- C -\frac{\log m_\beta}{\beta}.$$
Inserting into \eqref{eqhmub} we deduce that 
\be\label{tu}
\log \mub=  \beta(c_\beta - h^{\mub}) - \beta V\le \beta  \max(1, \log |x|)  \indc_{\d=2}+\beta C +\log m_\beta - \beta V.\ee In view of \eqref{assumpV2}, there thus exists $R>0$ independent of $\beta$ such that if $x\in \Rd\backslash B_R$ we have  $\log \mub< \log m_\beta -1$ (and recall that $m_\beta<\infty$ by Lemma \ref{lem22}). This, with the fact that $\mub$ is continuous, implies that $\sup_{\R^\d} \mub$ must be a maximum, and it must be achieved at some point 
 $x_\beta $ in $B_R$. Then we must  have $\Delta \log \mub(x_\beta) \le 0$,  hence by \eqref{eqmub}
$$\cd \mub(x_\beta)-\Delta V(x_\beta) \le 0.$$
We may then deduce that 
$$m_\beta = \mub (x_\beta) \le \frac{1}{\cd} \max_{B_R} \Delta V$$
i.e. that $m_\beta $ is bounded independently of $\beta$. The first bounds in the right-hand side of \eqref{bornelmub} follow.
The bound in \eqref{tu} then gets improved to 
$$\log \mub \le  \beta  \max(1, \log |x|)  \indc_{\d=2}+\beta C - \beta V$$
which yields the second set of bounds in \eqref{bornelmub}.

\end{proof}

\begin{prop}
\label{lemcomp}
There exists $C>0$ (depending only on $V$ and $\d$) such that if $\beta$ is large enough, we have 
\be \label{compp}
- \frac{C}{\beta} \le h^{\mub}-c_\beta  - ( h^{\mu_\infty}- c_\infty)\le  \frac{C \log \beta}{\beta} .\ee
\end{prop}
\begin{proof}
The lower bound is  an immediate consequence of \eqref{compp0} and \eqref{bornelmub}. Let us turn to the upper bound.

 We know that 
$$\min(h^{\mub}-c_\beta+V, -\Delta h^{\mub}) = \min \(-\frac1\beta \log \mub, \mub\).$$
If the right-hand side were $\le 0$ we could directly conclude by  comparison principle.  Instead,  we need to modify our test function slightly.
To that end, let us define 
$$E:= \{x\in \R^\d: \mub(x) < \beta^{-2}\}.$$
Let us estimate $\mub(E)$: using \eqref{bornelmub} and \eqref{assumpV3},  we find that 
\begin{align}\label{alig}  \mub(E) &\le C \int_E \beta^{-1} \( \exp \( - \frac{\beta}{2}(V-C) \) \wedge 1\) \\
\notag & \le C \beta^{-1}\end{align}
or respectively using $V - \log |x|-C$ in dimension $2$.
Since  $\mub(\R^\d)=1$ and $\mub \le C$, it also follows that if $\beta$ is large enough,
\be \label{214}|\R^\d\backslash E| \ge \frac{1}{C}.\ee

Let now $w$ be 
\be w:=\g* \(\mub \indic_E -\frac{ \mub(E) }{|\Rd \setminus E|}\indic_{\Rd \setminus E}\) \ee
  This way $w$ decays  like $|x|^{1-\d}$ in  all dimensions $\d \ge 2$, and in view of \eqref{alig}  we have 
\be\label{bornew}
\forall x \in \R^\d, \quad |w(x)| \le C \beta^{-1}.\ee Let us then set
$$v:= h^{\mub}-c_\beta+ c_\infty - \frac{2}{\beta}\log \beta - w- C \beta^{-1}$$
for the $C$ of \eqref{bornew}.
Observe that  
\begin{equation*}
-\frac{1}{\cd}\Delta v = \mu_\beta \indic_{\Rd \setminus E}  + \frac{ \mub(E) }{|\Rd \setminus E|}\indic_{\Rd \setminus E}     \quad \mbox{in} \ \Rd. 
\end{equation*}
By choice of $E$, \eqref{eqhmub} and \eqref{bornew}, we have in $\Rd \setminus E$,
\begin{align}
\label{align}
v+V-c_\infty& = h^{\mub}+V - c_\beta- \frac{2}{\beta} \log \beta -w - C \beta^{-1}
\\ \nonumber
& = -\frac1\beta \log \mub - \frac{2}{\beta} \log \beta -w - C \beta^{-1} \le 0.
\end{align} It follows that 
\be %\label{supersolu}
\min(v+V-c_\infty, -\Delta v) \le 0  
\quad \mbox{in} \ \Rd.\ee
In dimension $\d=2$ the comparison principle of Lemma \ref{l.comp} allows to conclude that $ v \le h^{\muv}$ which yields the desired upper bound for $h^{\mub}$. 
Let us now turn to dimension $\d \ge 3$.
Setting 
$$\varphi:= h^{\muv} -v,$$ by \eqref{align} and \eqref{op}
we have  
\be \label{frt}
\left\{\begin{array}{ll}
\varphi\ge 0 & \text{ in} \ \Rd \setminus E\\
-\Delta\varphi\ge 0 & \text{ in} \  E.\end{array}\right.\ee
 We also have 
$ \varphi \to c_\beta -c_\infty+\frac{2}{\beta}\log \beta +C \beta^{-1} $ at $\infty.$
Arguing as in the proof of Lemma \ref{lem24}, let $\psi$ be a harmonic function equal to zero on $\pa E$ and $c_\beta -c_\infty+\frac{2}{\beta}\log \beta+ C \beta^{-1} $ at infinity, we have $ \varphi \ge \psi$ in $E$ and if $ c_\beta -c_\infty+\frac{2}{\beta}\log \beta+ C \beta^{-1}<0$, 
 $\psi $ tends to its limit from above at speed $|x|^{2-\d}$.
 On the other hand  $\int_{\Rd} \Delta \varphi =0$. As in the proof of Lemma \ref{lem24}, we get a contradiction 
  and conclude that $c_\beta -c_\infty+\frac{2}{\beta}\log \beta+ C \beta^{-1} \ge 0$. We then conclude from \eqref{frt} and the maximum principle that $\varphi \ge 0$ everywhere, which yields the desired result.  
 \end{proof}

We deduce some corollaries.
\begin{lem} There exists $C>0$ (depending only on $V$ and $\d$) such that 
\be \label{expdecay} 
  \mub(x)\le \exp\(-\beta\alpha \min (1,  \dist(x, \Sigma)^2) + C\) ,\ee
 \be \label{minosig} \mub(x) \ge \exp(-C\log \beta) \quad \text{for x in} \ \Sigma,\ee
 \be \label{massmuc}
\mub(\Sigma^c) \le \frac{C}{ \sqrt{\beta}},\ee
and 
\be \label{masslogmuc}
\left|\int_{\Sigma^c} \mub\log \mub \right| \le \frac{C}{\sqrt\beta}.\ee
 \end{lem}
 \begin{proof}
Taking the exponential  of \eqref{eqhmub} and using \eqref{defzeta}, we find 
\be\label{bmfin0}
\mut=\exp(  \beta ( c_\beta - h^{\mut}-V))=\exp\(\beta (c_\beta-c_\infty+h^{\muv}-h^{\mub}- \zeta)\).
\ee
Inserting \eqref{compp} and \eqref{assumpV5}, we find 
\be \label{ce4}
\exp\(-\beta \zeta (x)-  C \log \beta    \)\le
\mub(x) \le \exp\(-\beta \alpha \min (1, \dist(x, \Sigma)^2 ) + C\).\ee
The upper bound in  \eqref{expdecay} follows, and \eqref{minosig} as well since $\zeta=0$ in $\Sigma$.

For  \eqref{massmuc},  using \eqref{ce4} and  \eqref{bornelmub} as well as \eqref{assumpV3} and  the coarea formula, we   may write 
\begin{equation}\label{345} 
\mub(\Sigma^c) \le C \int_0^\alpha \exp(-\beta \alpha s^2) ds + C \exp\(-\frac{ \beta}{2} \alpha \)\int_{\R^\d}\exp\( - \frac{\beta}{2} (V-C) \wedge 1\) 
\le \frac{C}{\sqrt{\beta}} ,
\end{equation}
or respectively using $V - \log |x|-C$ in dimension $\d=2$.

Arguing in the same way, and using the behavior of the function $x \log x$  we have 
\begin{align*}
\left|\int_{\Sigma^c} \mub\log \mub \right|
&
\le 
C \int_0^\alpha \beta s^2 \exp(-\beta \alpha s^2) ds
\\ & \qquad  
+ C\beta \alpha\exp\(-\frac{ \beta}{2} \alpha \)\int_{\R^\d}\exp\( - \frac{\beta}{2} (V-C) \wedge 1\) 
\\ & 
\le \frac{C}{\sqrt \beta}
\end{align*} 
(respectively with $V-\log |x|$ in dimension 2) 
which proves \eqref{masslogmuc}.
\end{proof}

 %\be \label{minomub}
%\mub(x) \ge \exp\( - \beta C \dist (x, \Sigma)^2 - C \log \beta \).\ee

\section{Study of the radial case and barrier argument}\label{secrad}
 Here we  first specialize to $V(x)=\frac{\lambda}{2} |x|^2$, which will provide a barrier function for the general case.
The problem is then radial and the solution $\mub(x)=e^{u_\beta(|x|)}$ with $u_\beta$ solving in place of \eqref{eqhmub} the ODE
\be \label{ode}\frac{1}{r^{\d-1}} (r^{\d-1} u_\beta ')'= \beta(\cd e^{u_\beta}-\lambda).\ee
By scaling, the coincidence set $\Sigma$ is then a ball of radius $R_\d \lambda^{-1/\d}$, where $R_\d$ only depends on $\d$, more precisely
 $\muv=  \frac{1}{\cd}  \indic_{B(0, R_\d \lambda^{-1/\d}  )}$. 
We first note that at a point of local maximum of $\mub$ we have $\Delta \log \mub\le 0$ hence $\mub \le \frac{\Delta V}{\cd}= \frac{\lambda}{\cd}$.
We thus know that $\cd\mub\le \lambda$ everywhere and thus 
$(r^{\d-1} u_\beta')' \le 0$ and $r^{\d-1} u_\beta '\le 0$ hence $u_\beta$ is nonincreasing.

In view of the exponential decay proved for the general problem in \eqref{expdecay}, for $1\ll K_\beta  \le C \log \beta$, there exists an $r_2$ (depending on $\beta$)  and bounded above by $2 R_\d \lambda^{-1/\d}$  such that  $u_\beta(r_2)= - K_\beta$.

\begin{lem}\label{lem33}
Let  $\eta$ be such that $e^{-K_\beta } \le \eta\le \frac{\lambda}{2\cd}$, and  let $ r_2$  be as above  such that  $u_\beta(r_2) = -K_\beta $. There exists $r_1 \ge r_2 -C \sqrt{\frac{K_\beta +\log \eta}{\beta}}$ (depending on $\beta$) such that 
$$u_\beta(r_1) =\log \eta,$$ with $C$ depending only on $\d$ and $\lambda$.
\end{lem}
\begin{proof}Integrating \eqref{ode},  we may write 
\be \label{ode2} r^{\d-1}  u_\beta '(r)=\int_0^r \beta s^{\d-1} (\cd e^{u_\beta (s)}-\lambda) \, ds.\ee
Let $r_1 $ be the largest $r \le r_2$ such that $u_\beta(r_1) = \log \eta$.  For $r \ge r_1$ we have $u_\beta(r)   \le \log \eta \le \log \frac{\lambda}{2\cd}$ hence
\begin{align*} 
\log \eta+ K_\beta 
= u_\beta(r_1)-u_\beta(r_2)
&
= \beta \int_{r_1} ^{r_2 } \frac{1}{t^{\d-1}} \int_0^t s^{\d-1} (\lambda- \cd e^{u_\beta (s)}) \, ds\, dt 
\\ & 
\ge \beta
\frac{\lambda }{2}  \int_{r_1 }  ^{r_2 } \frac{1}{t^{\d-1}} \int_{r_1}^t s^{\d-1}  \, ds\, dt  
\\ & 
\ge \beta \frac{\lambda }{2 \d } \int_{r_1}^{r_2}\frac{ t^{\d}- r_1^{\d}}{t^{\d-1}} \, dt 
\\ & 
\ge  \beta \frac{\lambda }{2 \d } \int_{r_1}^{r_2}( t- r_1) \, dt  
\\ & 
\ge \beta \frac{\lambda }{4 \d }  (r_2-r_1)^2.
\end{align*}
The result follows.
\end{proof}

We may now use the radial solution as a barrier for the solution in the general case.
\begin{prop}\label{pro34} Let 
\be\label{defMb}
M_\beta := \beta \max_{\R^\d} (h^{\mub}-c_\beta- h^{\muv}+c_\infty),\ee
and $1\ge \eta\ge e^{-M_\beta}$. 
There exists $C>0$ depending only on $V $ and $\d$ such that  for $x\in \Sigma$ satisfying  $\dist (x, \partial \Sigma) \ge C
\sqrt{\frac{M_\beta}{\beta}}$,
  we have
\be \label{lbmu}
\mub(x) \ge \eta.
\ee
\end{prop}
\begin{proof} We know from \eqref{compp} that $M_\beta \le C \log \beta$. 
Taking the exponential  of \eqref{eqhmub} and using the definition \eqref{defzeta} and the definition of $M_\beta$, we find 
\be\label{bmfin01}
\mut=\exp(  \beta ( c_\beta - h^{\mut}-V))=\exp\(\beta (c_\beta-c_\infty+h^{\muv}-h^{\mub}- \zeta)\) \ge e^{-M_\beta}\ee in $\Sigma$, since $\zeta=0$ in $\Sigma$.
 Since $\pa \Sigma \in C^{1,1}$, it satisfies an interior ball condition,  with a ball of radius which can be chosen independently of the point, say of radius $\ep$. We then choose $\lambda\ge 2 \cd$ large enough that $\lambda \ge \alpha$ and $2 R_\d\lambda^{-1/\d}\le\ep$. Given this $\lambda$, we consider $\nu$ to be  $\alpha/\lambda$ times the  radial $\mu_{\frac{\alpha\beta}{\lambda}} $ of Lemma \ref{lem33}, which satisfies 
\be \label{eq112}\Delta \log \nu = \beta (\cd\nu - \alpha).\ee
We also let $K_{\alpha\beta/\lambda}=M_\beta$ in Lemma \ref{lem33} applied at the inverse temperature $\beta \alpha/\lambda$, and  we let $r$ be the $r_2$ given there. 
  Since $r_2 \le 2 R_\d \lambda^{-\frac1\d}\le \ep$, a ball  $B_r$ tangent to $\pa \Sigma$  at any point can be included in $\Sigma$.
 In view of  \eqref{bmfin01}, the monotonicity of $\nu$ and the definition of $r$ and $K_{\alpha \beta/\lambda}$,  we check that $\nu \le \mub$ on $\partial B_{r}$.
 We now substract \eqref{eqmub} and \eqref{eq112} and test the resulting relation against $(\log \nu - \log \mub)_+$ which vanishes on $\partial B_{r}$.
We obtain 
$$\int_{B_{r}} (\Delta \log \nu -\Delta \log \mub) (\log \nu - \log \mub)_+= \beta \int_{B_{r}} (\cd \nu-\cd \mub + \Delta V-\alpha) (\log \nu - \log \mub)_+.$$
Using that $\Delta V\ge \alpha$ in $B_{r}$ by \eqref{assumpV4} and an integration by parts, we are led to 
$$-\int_{B_{r}\cap \{\nu \ge \mub\} }| \nab (\log \nu - \log \mub)|^2 \ge \beta \cd \int_{B_{r}} (\nu-\mub) (\log \nu - \log \mub)_+\ge 0.$$
It follows that 
$\nu \le \mub$ a.e. in $B_{r}$, thus $\nu$ is a barrier for $\mub$. In view of the result of Lemma~\ref{lem33}, we deduce that 
$\mub\ge \eta$ as soon as $x\in B_{r}$ and $\dist (x, \partial B_{r})\ge  C
\sqrt{\frac{M_\beta+\log \eta}{\beta}}  $ for some $C$ depending only on $V$ and $\d$.
The result follows.
\end{proof}

\section{Optimal estimates and lower bound on \texorpdfstring{$\mub$}{mu beta}} \label{secopt}

We may now conclude
\begin{prop}
\label{lemcomp2}
There exists $C>0$ (depending only on $V$ and $\d$) such that if $\beta$ is large enough, we have 
\be \label{compp2}
 h^{\mub}-c_\beta  - ( h^{\mu_\infty}- c_\infty)\le  \frac{C }{\beta} ,\ee
\be \label{expdecay2} 
 \exp\( - \frac{\beta}{ C }\dist (x, \Sigma)^2- C\) \le \mub(x) \ee
for $x$ in a neighborhood of $\Sigma$, and
 \be \label{lbmb}
 \mub(x) \ge\frac1C>0 \quad \text{for} \ x\in \Sigma.\ee
\end{prop}

\begin{proof}[Proof of Proposition \ref{lemcomp2}] We iterate and improve on the proof of Proposition \ref{lemcomp}.
 Let $M_\beta$ be as in \eqref{defMb}. We know from \eqref{compp} that $M_\beta \le C \log \beta$.  If $M_\beta$ is bounded above independently of $\beta$, then there is nothing to prove. If $M_\beta \to +\infty$ as $\beta \to \infty$, let $\eta$ be a constant in $[ e^{-M_\beta},\min( \frac{\lambda}{2\cd}, 1)]$, to be determined later.  Let then $$\hat \Sigma= \left\{x \in \Sigma, \dist (x, \pa \Sigma ) \ge C \sqrt{\frac{M_\beta}{\beta}}:=\tau_\beta \right\}$$ with $C$ as in Proposition \ref{pro34}. In view of that proposition we have that $\mub\ge \eta$ in $\hat \Sigma$. 
 %We know that 
%$$\min(h^{\mub}-c_\beta+V, -\Delta h^{\mub}) = \min \(-\frac1\beta \log \mub, \mub\).$$
%We would like the right-hand side to be $\le 0$, so we need to modify our test function slightly.

Since $\pa \Sigma \in C^{1,1}$ so is $\pa \hat \Sigma$ for $\beta$ large enough, and  we may consider $R(x)$ to be the reflexion with respect to $\pa \hat \Sigma $, defined in a tubular neighborhood. We next let $\hat \mub= \mub \indic_{E} $ where 
$$E:= \( \{  0 \le \dist (x, \Sigma) < \beta^{-\frac14 }\} 
\cup (\Sigma \backslash \hat \Sigma) \)  \cap  \{ \mub(x) \le \eta\}.$$ Arguing as in \eqref{345} we have that \be \label{567} |
 \mub ( ( \Sigma\backslash E)^c) - \hat \mub (E)  
|< \frac{C}{\beta}.\ee
We note that $E$ is included in a $\tau_\beta+ \beta^{-1/4}$ neighborhood of $\pa \Sigma$.

Let  $w$ be 
\be w:=\g* \(\hat  \mub - R\#\hat \mub + ( \mub  \indic_{(\Sigma\backslash E)^c} - \hat \mub) 
 -\frac{ \mub ( ( \Sigma\backslash E)^c) - \hat \mub (E)    }{|\hat \Sigma|}\indic_{\hat \Sigma}\), \ee
 where $\# $ denotes the push-forward of measures.
 We claim that
  \be\label{bornew1}
\forall x \in \R^\d, \quad |w(x)| \le C\frac{\eta}{\beta}( M_\beta+ |\log \eta|),\ee with $C$ depending only of $V$ and $\d$. 
 We note that $w$ decays  like $|x|^{1-\d}$ in  all dimensions $\d \ge 2$ because its Laplacian integrates to $0$, and in view of \eqref{567} we have 
  $$\left|\g*\(   \mub  \indic_{(\Sigma\backslash E)^c} - \hat \mub 
 -\frac{ \mub ( ( \Sigma\backslash E)^c) - \hat \mub (E)  }{|\hat \Sigma|}\indic_{\hat \Sigma}\) \right|\le C \beta^{-1}
 .$$ 
 Hence there remains to show that  
 \be \label{ghj} 
 |\g*  \(\hat  \mub - R\#\hat \mub \) |\le C\eta \frac{M_\beta}{\beta}.\ee 
By definition of the push-forward, we have 
\begin{align}
\label{3term}
|\g*  \(\hat  \mub - R\#\hat \mub \)(x) | 
&
= \left|\int \( \g(x-y)-\g(x-R(y)) \) d\hat \mub (y)\right|
\\ \notag &
\leq C\eta \int_{\min (|y-x|, |R(y)-x|) \le \tau_\beta    } |\g(x-y)-\g(x-R(y))| dy 
\\ \notag & \qquad 
+ C \int_{|y-x|\ge\tau_\beta , |R(y)-x| \ge\tau_\beta }  
\frac{|\langle y-R(y), y-x\rangle|}{|y-x|^{\d}} d\hat \mub(y) 
\\ \notag & \qquad 
+ C \int_{|y-x|\ge \tau_\beta   , |R(y)-x| \ge \tau_\beta} 
\frac{|y-R(y)|^2}{|y-x|^{\d}} d\hat \mub(y).
\end{align}
 The first time in the right-hand side is bounded by $C \eta\tau_\beta^2$ in all dimensions, while the second is bounded by 
 \begin{multline}\label{sect} C \int_{|y-x|\ge\tau_\beta  ,0< \dist(y, \Sigma) < \beta^{-1/4}}  
 \frac{\dist(y,\hat \Sigma)  |\langle y-x, n(y)\rangle|} {|y-x|^{\d}}\min \( \eta,  \exp(-\beta \alpha\,  \dist^2(y, \Sigma))\) \, dy\\
 +C \eta  \int_{|y-x|\ge \tau_\beta , y \in \Sigma \backslash \hat \Sigma}  
 \frac{\dist(y,\hat \Sigma)  |\langle y-x, n(y)\rangle|} {|y-x|^{\d}}  \, dy
\end{multline} 
where~$n(y)$ is the normal vector to~$\pa \hat\Sigma$ at~$y$. We claim that these terms are bounded  by $C \eta \tau_\beta^2$. 
To see this, let us use local coordinates adapted to $\hat \Sigma$ of the form $(s, t) \in\pa \hat \Sigma \times \R$ such that each vector $y$ can be decomposed into $y=s+t n(y)$, with $s \in \pa \hat \Sigma$ and $t=\dist (y,\hat \Sigma)$.  We choose the origin so that $x$ has coordinates $(0, t_0)$, and assume $|t_0|<\ep$ (for otherwise the result is clearly true).  The Jacobian of the change of variables is bounded, and $\langle y-x, n(y)\rangle = O(|s|^{2} ) + t-t_0$ as $|s| \to 0$,  using that $\pa \hat \Sigma \in C^{1,1}$,  so we may locally  on $\pa\hat  \Sigma$  bound this integral by 
\begin{multline*}
 C \int_{|t-t_0| \ge \tau_\beta  ,\tau_\beta  <  t  <  \tau_\beta + \beta^{-1/4}, s \in \pa\hat \Sigma } 
\frac{t\min \(\eta, \exp(-\beta  \alpha (t-\tau_\beta)^2)\) ( (t-t_0)+O(s^{2} ) )  }{ \( |t-t_0|^2 + |s|^{2} + O( |s|^{2} |t-t_0|) \)^{\frac\d2} }ds \, dt \\
\le C\eta  \int_{ \tau_\beta <t<\tau_\beta+ \beta^{-1/4} , s' \in \pa\frac{\hat \Sigma}{|t-t_0|} } \frac{t  \min\(\eta,  \exp(-\beta \alpha (t-\tau_\beta)^2)\)   (O(1+ |s'|^{2}  |t-t_0|) )   }{ \hal \( 1 + |s'|^2  \)^{\frac\d2} }ds' \, dt\\ \le C\eta \tau_\beta \sqrt{\frac{|\log \eta|}{\beta}} \le C \( \eta \tau_\beta^2 + \frac{\eta|\log \eta|}{\beta}\)
\end{multline*} where used the change of variables $s= |t-t_0|s'$, that $\pa \hat \Sigma \in C^{1,1}$ and that $\ep$ can be  chosen small enough.

The  second term in \eqref{sect}   and the third term in \eqref{3term} are bounded by $C \eta \tau_\beta^2$  by  similar computations, which 
concludes the proof of  \eqref{ghj} hence of \eqref{bornew1}. 
  Let us then set
$$v:= h^{\mub}-c_\beta+ c_\infty - w- C \frac{\eta}{\beta} ( M_\beta+ |\log \eta|)  + \frac{\log \eta}{\beta}$$
for the $C$ of \eqref{bornew1}.
Observe that  
\begin{equation*}
-\frac{1}{\cd}\Delta v = \mu_\beta \indic_{\Sigma\backslash E}  + R\# \hat \mub + \frac{ \mub(\Sigma^c) -\hat \mub(\Sigma^c)}{|\hat\Sigma|}\indic_{\hat \Sigma}     \quad \mbox{in} \ \Rd,
\end{equation*}
hence $\Delta v$ is supported in $\Sigma\backslash E$.
By  \eqref{eqhmub},  \eqref{bornew1} and $\mub \ge \eta$ in $ \Sigma\backslash E$, we have in $\Sigma \backslash E$,
\begin{align}
\label{align1}
v+V-c_\infty& = h^{\mub}+V - c_\beta -w- C \frac{\eta}{\beta} ( M_\beta+ |\log \eta|)+ \frac{\log \eta}{\beta}
\\ \nonumber
& = -\frac1\beta \log \mub+ \frac{\log \eta}{\beta} - C \eta \frac{M_\beta+|\log \eta|}{\beta}
  -w  \le 0.
\end{align} It follows that 
\be %\label{supersolu}
\min(v+V-c_\infty, -\Delta v) \le 0  
\quad \mbox{in} \ \Rd.\ee
In dimension $\d=2$ the comparison principle of Lemma \ref{l.comp} allows to conclude that $ v \le h^{\muv}$ which yields the desired upper bound for $h^{\mub}$. 
Let us now turn to dimension $\d \ge 3$.
Setting 
$$\varphi:= h^{\muv} -v,$$ by \eqref{align1} and \eqref{op}
we have  
\be \label{frt1}
\left\{\begin{array}{ll}
\varphi\ge 0 & \text{ in} \ \Sigma \backslash E\\
-\Delta\varphi\ge 0 & \text{ in} \  (\Sigma \backslash E)^c.\end{array}\right.\ee
 We also have 
$ \varphi \to c_\beta -c_\infty
- C \eta \frac{M_\beta+|\log \eta|}{\beta}+ \frac{\log \eta}{\beta}
 $ at $\infty.$
Arguing as in the proof of Lemma~\ref{lem24}, let $\psi$ be a harmonic function equal vanishing on $\pa (\Sigma \backslash E)$ and $c_\beta -c_\infty
- C \eta \frac{M_\beta+|\log \eta|}{\beta}+ \frac{\log \eta}{\beta} $ at infinity, we have $ \varphi \ge \psi$ in $\Sigma\backslash E$ and if $ c_\beta -c_\infty - C \eta \frac{M_\beta+|\log\eta|}{\beta} + \frac {\log \eta}{\beta}<0$, 
 $\psi $ tends to its limit from above at speed $|x|^{2-\d}$.
 On the other hand  $\int_{\Rd} \Delta \varphi =0$. As in the proof of Lemma~\ref{lem24}, we get a contradiction 
  and conclude that $c_\beta -c_\infty
  - C \eta \frac{M_\beta+|\log \eta|}{\beta}+ \frac{\log \eta}{\beta} \ge 0$. We then conclude from \eqref{frt1} and the maximum principle that $\varphi \ge 0$ everywhere, which yields 
  \begin{equation*} 
  h^{\muv} -c_\infty - h^{\mub}+c_\beta \ge     
    \frac{\log \eta}{\beta} - C \eta \frac{M_\beta+|\log \eta|}{\beta}
 \end{equation*} hence by definition of $M_\beta$
 $$\frac{M_\beta}{\beta} \le  C \eta \frac{M_\beta+|\log \eta|}{\beta}- \frac{\log \eta}{\beta}.$$
 Choosing $\eta$  a small enough constant, we obtain that $M_\beta \le C$, which  concludes the proof of \eqref{compp2}. The result \eqref{expdecay2} 
   follows from combining \eqref{bmfin0}  and  \eqref{compp2}, and  \eqref{lbmb}  follows from combining \eqref{bmfin0},   \eqref{compp2} and the  fact that $\zeta=0$ in $\Sigma$.
 \end{proof}

\begin{coro}
We have 
\be \label{bornegrad}\|\nab (h^{\mub}-h^{\mu_\infty})\|_{L^\infty(\R^\d)}\le  C \beta^{-\hal}
.\ee
\end{coro}
\begin{proof}
This follows from \eqref{compp}, the fact that $\|\mub-\muv\|_{L^\infty} \le C$ and interpolation (see for instance the appendix in \cite{bbh2}).
\end{proof}

\section{Regularity theory and iterative approximation}\label{sec4}
Once $\mub$ is bounded below, the PDE \eqref{divfor} becomes uniformly  elliptic and we may apply regularity theory tools to compare $\mub$ to the expected solution.
In the case that $\Delta V$ is constant, then we can show that $\mub$ is very close to the constant $\muv$  inside $\Sigma$, however in the case where $\Delta V$ is not constant, there are corrections to arbitrary order that need to be added to $\muv$.

Assuming that $V \in C^{2m,\gamma}$ for some~$m\in\mathbb{N}$ and exponent~$\gamma>0$, we recursively define $f_k$ by \eqref{41}.
We note that, for $\beta$ sufficiently large depending on  the norms of $V$ and on $k$, and by \eqref{assumpV4},
\begin{equation}
\label{e.fkbounds}
\|f_k\|_{C^{2(m-k-1),\gamma}(\Sigma)}
\leq C 
\quad \mbox{and} \quad 
f_k \geq \frac{\alpha}{4\cd} \quad \mbox{in} \ \Sigma. 
\end{equation}
We also define 
\be\label{epk}
\ep_k := \Delta \log f_k - \beta (\cd f_k- \Delta V)= \beta \cd (f_{k+1}-f_k)
\ee
and check that
$$\ep_{k+1}= \Delta \log \( 1+ \frac{\ep_k}{\beta \cd f_k}\)$$
and  thus
\begin{equation}
\label{epkbounds}
\|\ep_k\|_{C^{2(m-k-2),\gamma}(\Sigma)} \leq
C\beta^{-k}.
\end{equation}
Thus since $\ep_k$ gets small as $\beta $ gets large,  $f_k$ is a good approximate solution to \eqref{eqmub} for $k\ge 1$.  In view of \eqref{epk} and \eqref{epkbounds}, if $V\in C^\infty$ then $f_k$ converges as $ k \to \infty$ in all $C^m$ spaces  to $f_\infty$, an exact solution of \eqref{eqmub}.
\begin{prop}
Assume $m\in\mathbb{N}$, $m\ge 2$, and $\gamma\in (0,1]$ are such that $V \in C^{2m,\gamma}$. Then  for every $n$ even integer with $n \le 2(  m-2)$ and every $0\le \gamma'\le \gamma$, there exists $C>0$ depending only on $V,\d, n$ 
such that if $\beta$ is large enough depending on $m$,
\begin{equation}
\label{e.higherreg}
 \left\| \mub(x)-  f_{m-2-\frac{n}{2}}(x)\right\|_{C^{n, \gamma'}( \Sigma)} 
\leq 
C\beta^{\frac{n+\gamma'}{2}}  \exp\left( - C \log^2 \bigl(\beta \dist^2\big(x, \partial \Sigma\big) \big)  \right) 
+ C \beta^{n-m+1+\frac{\gamma'}{2}}. \end{equation}
\end{prop}
\begin{proof}
Define $u_\beta :=\frac{\mub}{f_k}- 1$. %where $f_k$ is as above (and thus bounded below in $\Sigma$). 
By~\eqref{eqmub}, we have
$$\Delta \log (f_k (u_\beta+ 1)) =\beta(\cd f_k  u_\beta+ \cd f_k-\Delta V).
$$
In view of~\eqref{epk}, we get 
$$\Delta \log (1+u_\beta ) = \beta \cd f_k u_\beta- \ep_k$$
which can be rewritten as 
\be \label{divform}
- \div \( \frac{\nab u_\beta}{1+u_\beta}\) +  \beta \cd f_k u_\beta = \ep_k.
\ee
This equation is uniformly elliptic in~$\Sigma$ since, by \eqref{bornelmub},\eqref{lbmb} and~\eqref{e.fkbounds}, 
\begin{equation}
\label{ubetabounds}
\frac\alpha C \leq u_\beta +1 \leq \frac C\alpha \quad \text{in} \ \Sigma. 
\end{equation}

%According to Proposition \ref{pro34}, we have that 
%$\mub \ge \frac{\alpha}{2} $ in $\hat \Sigma$.
%Let $B(x,r)$ be a ball included in $\hat \Sigma$. Let 
%$\chi$ be a nonnegative cutoff function equal to $1$ in $B(x,\hal r)$, vanishing outside $B(x, r)$ and such that $|\nab \chi |\le Cr^{-1}$.
%Let us consider  We note that $u_\beta +1 $ is bounded above and below by positive constants independent of $\beta$.

%\cm{we must write what follows as a general lemma for equations of the form 
%$$\div \( \frac{\nab u}{1+u}\) = \beta \mu u + \ep$$
%with $\mu$ bdd below.}

\smallskip

We next seek a local $L^2$ estimate for $u_\beta$ in $\Sigma$. Select~$x_0\in \Sigma$, $r\in \big( 0 ,\dist \big(x_0,\partial \Sigma\big) \big)$ and a cutoff function $\chi\in C^\infty_c(B_r)$. Testing~\eqref{divform} with~$\chi^2 u_\beta$, we obtain
$$
\int_{B_r(x_0)} \chi^2 \frac{|\nab u_\beta|^2 }{1+u_\beta}
+ \int_{B_r(x_0)} \frac{2u_\beta \chi  \nab \chi \cdot \nab u_\beta}{1+u_\beta} +  \beta \cd \int_{B_r(x_0)} \chi^2 f_k u_\beta^2 
= \int_{B_r(x_0)} \chi^2 \ep_k u_\beta .
$$
Using Young's inequality and~\eqref{e.fkbounds}, we obtain after rearrangement that
\begin{align*}
\frac 12\int_{B_r(x_0)} \chi^2 \frac{|\nab u_\beta|^2 }{1+u_\beta}
+
\frac{\beta \cd}2  \int_{B_r(x_0)} \chi^2 f_k u_\beta^2
&  
\leq
4 \int_{B_r(x_0)} \frac{u_\beta^2 \left| \nabla \chi \right|^2}{1+u_\beta}
+
\frac1{2\beta \cd} \int_{B_r(x_0)}\chi^2 \ep_k^2
\\ & 
\leq 
C\int_{B_r(x_0)} |\nab \chi|^2 u_\beta^2 
+ \frac{C}{\beta} \int_{B_r(x_0)} \ep_k^2.
\end{align*}
Choosing~$\chi$ such that $\indc_{B_{r/2}(x_0)} \leq \chi \leq \indc_{B_{r}(x_0)}$ and~$|\nab \chi | \leq 4r^{-1}$ and using~\eqref{e.fkbounds}, ~\eqref{epkbounds} and~\eqref{ubetabounds}, we find that, if $k\le m-2$, then
\begin{equation}
\dashint_{B_{r/2}(x_0)} \left| \nabla u_\beta \right|^2 
+
\beta \dashint_{B_{r/2}(x_0)} u_\beta^2
\leq
\frac{C}{r^2} \dashint_{B_{r}(x_0)} u_\beta^2 + C\beta^{-(2k+1)}
\end{equation}
In particular, keeping only the second term on the left side, we obtain 
\begin{equation}
\label{toiter}
\dashint_{B_{r/2}} u_\beta^2 \le \frac{C}{\beta r^2} \dashint_{B_r} u_\beta^2  + C\beta^{-2(k+1)}.
\end{equation}
%\sa{Before there was $k$ instead of $2k$ in~\eqref{toiter}, so what is below needs to change to reflect that.}
%Iterating~\eqref{toiter}, we find that 
%\begin{align*}
%\dashint_{B_{2^{-j} r}(x_0)} u_\beta^2 
%& 
%\leq  
%\frac{C}{\beta^j r^{2j}} \prod_{l=0}^j 2^{2l} \dashint_{B_r(x_0)} u_\beta^2 +  \frac{C}{\beta^{k+2}} \sum_{q=1}^j \frac{C}{\beta^{j-q} r^{2(j-q)} }\prod_{l=q}^j 2^{2l} + \frac{C}{\beta^{k+2}} 
%\\ & 
%\leq   
%\frac{C}{\beta^j r^{2j}  } 2^{j(j+1)} \dashint_{B_r(x_0)} u_\beta^2 + \frac{C}{\beta^{k+2}} \sum_{q=1}^j \frac{C}{\beta^{j-q} r^{2(j-q)} }    2^{j(j+1)-  q(q-1)}   + \frac{C}{\beta^{k+2}}
%\\ & 
%\leq  
%\frac{C}{\beta^j r^{2j}  } 2^{j(j+1)}+ \frac{C}{\beta^{k+2}}  \sum_{s=0}^{j-1} \frac{2^{(1+s)(2j-s)}}{\beta^s r^{2s}} + \frac{C}{\beta^{k+2}}
%\\ & 
%\leq  
%\frac{C}{\beta^j r^{2j}  } 2^{j(j+1)}+ \frac{C}{\beta^{k+2}}  \frac{2^{2j^2 +2j}}{\beta^j r^{2j}} + \frac{C}{\beta^{k+2}}.
%\end{align*}
%We apply this with $j$ being the largest integer such that $C/(\beta 2^{-2j}r^2) \leq \frac12$, that is, $j: =c \log (\beta r^2)$, with $c$ suitably chosen. 
After an iteration of the previous inequality, we obtain, for $s:= C\beta^{-\frac12}$, 
\begin{align*}
\dashint_{B_{s}(x_0)} u_\beta^2 
\leq 
\exp\left( - c \log^2 \bigl(\beta r^2\big)  \right) 
+ C\beta^{-2(k+1)}.
\end{align*}
Let us now rescale the equation~\eqref{divform} by defining
\begin{equation}
\hat{u}_\beta(x):= u_\beta( x_0 +sx ),
\end{equation} and similarly $\hat f_k, \hat \ep_k$.
In terms of $\hat{u}_\beta$, the equation becomes
\begin{equation}
- \div \( \frac{\nab \hat{u}_\beta}{1+\hat{u}_\beta}\) +  \cd \hat f_k \hat{u}_\beta = \beta^{-1} \hat \ep_k \quad \mbox{in} \ B_1. 
\end{equation}
Note that the function $\cd \hat f_k$ is bounded. 
Applying the De Giorgi-Nash H\''older estimate (see~\cite[Theorem 8.24]{GT}) for uniformly elliptic equations, we obtain, for some $\sigma>0$ and again for $k\le m-2$,
\begin{align*}
\left\| \hat{u}_\beta \right\|_{L^\infty(B_{1/2})}
+ \left[ \hat{u}_\beta \right]_{C^{0,\sigma}(B_{1/2})} 
&
\leq
C \left(  \int_{B_1} \hat{u}_\beta^2 \right)^{\frac12}  + C\beta^{-1} \left\|\hat  \ep_k \right\|_{L^\infty(B_s)}
\\ & 
\leq
C \exp\( - c(\log^2 \bigl (\beta r^2\big ) \) +  C\beta^{-(k+1)}.
\end{align*}
Repeatedly applying Schauder estimates yields, for every $n \leq 2(m-k-2)$ and $0\le \gamma'\le \gamma$, 
\begin{equation} \left[ \nabla^n \hat{u}_\beta \right]_{C^{0, \gamma'}(B_{1/2})} 
\leq 
C \exp\( - c(\log^2 \bigl (\beta r^2\big ) \) +  C\beta^{-(k+1)},
\end{equation}
with constants $C$ which now depend on $m$. Taking $k=m-2-\frac{n}{2}$, 
after rescaling back, by definition of $u_\beta$  this implies~\eqref{e.higherreg}. \end{proof}
\begin{remark}
The estimates above apply in the same way to all solutions of relations of the form \eqref{divform}. This allows to handle questions of stability of the solutions with respect to $V$: if $V$ is changed into $V + t \xi$ with $\xi  $ supported in $\Sigma$, then letting $\mu_\beta^t$ be the corresponding thermal equilibrium measure, the function 
$u_t= \frac{\mub^t}{\mub}-1$ satisfies 
\be \label{divform2}
\div \( \frac{\nab u_t}{1+u_t}\) = \beta \(  \cd \mub^0 u_t-t\Delta \xi\) .\ee
which 
is of the same form as \eqref{divform}. The same method then allows to estimate $u_t$ hence $\mub^t/\mub$.
\end{remark}

%\section{Iterative approximation outside the contact set}
%We now introduce the other iterative approximation of $\mub$ outside $\Sigma$. Let 
%$$\phi_0  = \exp\(\beta a_0\) \exp\(- \beta (V+h^{\mu_\infty}- c_\infty)\) $$
%where $a_0$ is such that $\int \phi_0=1$.
%We note that $V+h^{\mu_\infty}- c_\infty=0$ in $\Sigma$ so that $z_0$ is bounded below by a constant independent of $\beta$, hence so is $\phi_0$. 
%We then let 
%$$\phi_{k+1} = \exp\(\beta c_k\) \exp\( - \beta (V+ h^{\phi_k})\).$$
%We  let as before 
%$$\eta_k:= \Delta \log \phi_{k+1} - \beta(\cd \phi_{k+1} - \Delta V) = \beta \cd(\phi_{k}-\phi_{k+1})$$
%We may study $h^{\phi_{k+1}}- c_{k+1} -h^{\phi_k}+c_k $ by maximum principle arguments as in the other section, using 
%$$ h^{\phi_k}+  V+ \frac{1}{\beta} \log \phi_{k+1}= c_k
%$$
%and subtracting off.
%Hopefully we can prove it remains $\le C/\beta$ and thus 
%$\phi_{k+1}/\phi_k$ remains bounded above and below and $\phi_k /\mub$ also.

\end{document}